\documentclass[12pt]{amsart}

\usepackage[T1]{fontenc}
\usepackage[utf8]{inputenc}

\usepackage{amsmath,amssymb,amsfonts,amsthm,mathtools}

\usepackage[a4paper,top=3cm,bottom=2cm,left=3cm,right=3cm,marginparwidth=1.75cm]{geometry}

\usepackage{graphicx}
\usepackage{faktor}
\usepackage{csquotes}
\usepackage{enumitem}
\usepackage{etoc}

\usepackage{listings}   
\usepackage{xcolor}     
\usepackage{upquote}

\usepackage[colorinlistoftodos]{todonotes}

\usepackage[normalem]{ulem}

\usepackage{url}
\usepackage[colorlinks=true, allcolors=blue]{hyperref}

\usepackage[
style=alphabetic,
hyperref=true,
isbn=false,
backend=bibtex,
maxnames=10,
doi=false,
url=false,
eprint=false,
giveninits=true,
]{biblatex}
\definecolor{codegreen}{rgb}{0,0.6,0}
\definecolor{codegray}{rgb}{0.5,0.5,0.5}
\definecolor{codepurple}{rgb}{0.58,0,0.82}
\definecolor{backcolour}{rgb}{0.95,0.95,0.92}

\lstdefinestyle{sagestyle}{
	backgroundcolor=\color{backcolour},
	commentstyle=\color{codegreen},
	keywordstyle=\color{magenta},
	numberstyle=\tiny\color{codegray},
	stringstyle=\color{codepurple},
	basicstyle=\ttfamily\footnotesize,
	breakatwhitespace=false,
	breaklines=true,
	captionpos=b,
	keepspaces=true,
	numbers=none,
	numbersep=5pt,
	showspaces=false,
	showstringspaces=false,
	showtabs=false,
	tabsize=2
}

\theoremstyle{plain} 
\newtheorem{theorem}{Theorem}[section]
\newtheorem*{theorem*}{Theorem}

\newtheorem{corollary}[theorem]{Corollary}
\newtheorem{lemma}[theorem]{Lemma}
\newtheorem{proposition}[theorem]{Proposition}

\theoremstyle{remark}
\newtheorem{definition}[theorem]{Definition}
\newtheorem{remark}[theorem]{Remark}

\def\ee{\mathbf{e}}

\def\vv{\mathbf{v}}

\def\ww{\mathbf{w}}

\def\CC{\mathbb{C}}

\def\NN{\mathbb{N}}
\def\QQ{\mathbb{Q}}
\def\RR{\mathbb{R}}
\def\ZZ{\mathbb{Z}}

\def\calC{\mathcal{C}}
\def\calD{\mathcal{D}}

\def\calM{\mathcal{M}}
\def\calo{\mathcal{O}}

\def\calT{\mathcal{T}}

\DeclareMathOperator{\CFF}{CFF}

\DeclareMathOperator{\SL}{SL}

\title[On $\mathfrak{P}$-adic continued fractions with extraneous denominators]{On $\mathfrak{P}$-adic continued fractions with extraneous denominators: some explicit finiteness results}
\author{Laura Capuano}
\address[Laura Capuano]{Dipartimento di Matematica e Fisica, Università degli Studi di Roma Tre}
\email{laura.capuano@uniroma3.it}
\author{Sara Checcoli}
\address[Sara Checcoli]{Univ. Grenoble Alpes, CNRS, IF, 38000 Grenoble, France}
\email{sara.checcoli@univ-grenoble-alpes.fr}
\author{Marzio Mula}
\address[Marzio Mula]{Research Institute CODE, Universität der Bundeswehr München}
\email{marzio.mula@unibw.de}
\author{Lea Terracini}
\address[Lea Terracini]{Dipartimento di Informatica, Università di Torino}
\email{lea.terracini@unito.it}

\date{\today}
\begin{document}

\maketitle
\begin{abstract}
  Let $K$ be a number field. We show that, up to allowing a finite set of denominators in the partial quotients, it is possible to define algorithms for $\mathfrak P$-adic continued fractions satisfying the finiteness property on $K$ for every prime ideal $\mathfrak P$ of sufficiently large norm. This provides, in particular, a new algorithmic approach to the construction of division chains in number fields.

\end{abstract}

\section{Introduction}
Let $K$ be a field. A (finite) continued fraction in $K$ is an expression of the form 
\[[a_0,a_1,\ldots,a_n]=
{{a_0 + \cfrac{1}{{a_1 +  \cfrac{1}{{\ddots + \cfrac{1}{{a_{n}}}}}}}}}, \hbox{ with } a_1,\ldots, a_n\in K.
\]
The elements $a_0,a_1,\ldots$ are called the \emph{partial quotients} of the continued fraction.\\
When $K$ is endowed with a topology,  one can also define infinite continued fractions as
\[[a_0,a_1,\ldots,a_n,\ldots]=\lim_{n\to\infty} [a_0,a_1,\ldots,a_n],\]
provided that this limit exists in the completion of $K$.
Continued fractions are a classical and fascinating topic which has been investigated in many branches of mathematics and beyond, ranging from Diophantine approximation~\cite{AndreescuAndrica:CFDioph} to the study of astronomy~\cite[§4.1]{rockett1992continued} and the tuning of musical instruments~\cite{dunneMcConnes:pianosCF}.\par

An approach to generating continued fraction expansions involves an iterative procedure that employs a \emph{floor function} $s:K\to K$. Given $\alpha \in K$, a typical algorithm  works as follows: 
\begin{equation}
    \label{eqn:CFalg}
    \begin{cases} \alpha_0 =\alpha\\
a_n =s(\alpha_n) \\
    \alpha_{n+1} = \frac 1 {\alpha_n-a_n} \quad \textrm{if } \alpha_n-a_n \neq 0,\end{cases} 
\end{equation}
and it stops if $\alpha_n=a_n$.
In this case, $a_n$ only depends on $\alpha_n$. However, there are examples in the literature of more complex algorithms taking into account also previous steps of the computation (see for example the second algorithm proposed in \cite{Browkin2000}). 

In the classical case, $K$ is the real field $\RR$ and the floor function is nothing else than the integral part of a real number. For classical real continued fractions there are many results that relate algebraic properties of real numbers to the form of their continued fraction expansion. For example, the Euclidean algorithm shows that rational numbers correspond to finite continued fractions, and Lagrange's theorem characterizes quadratic irrationals as those real numbers having a periodic continued fraction expansion.

Since the 1970s, several different definitions of $p$-adic  continued fractions  were proposed, for example in \cite{Ruban1970, Schneider1970, Browkin1978, Browkin2000}, and many authors studied the finiteness and periodicity properties in the $p$-adic setting which seem to be very different with respect to the real case (see for example \cite{Bedocchi1990, ooto2017:padicTrasc, CapuanoVenezianoZannier2019, CapuanoMurruTerracini2020, BarberoCerrutiMurru2021, MurruRomeo2023}).
\medskip

Recently, in \cite{CapuanoMurruTerracini2022}, the authors proposed a unifying approach for continued fractions over the non-Archimedean completions of a number field $K$, based on the notion of \emph{type}. 

A type is a triple $(K,\mathfrak{P},s)$ where $K$ is a number field, $\mathfrak{P}$ is a prime ideal of $K$,  and $s$ is a {$\mathfrak{P}$-adic floor function} for $K$, that is a function
$s$ on $K_{\mathfrak{P}}$ (the $\mathfrak{P}$-adic completion of $ K$) satisfying analogous properties of the classical real floor function. {More precisely, following \cite[Definition 3.1]{CapuanoMurruTerracini2022}, we say that a \emph{$\mathfrak{P}$-adic floor function} for $K$ is a function 
$s : K_{\mathfrak{P}} \to K$ such that
\begin{enumerate}[label=\textnormal{(\alph*)}]
    \item $|\alpha - s(\alpha)|_{\mathfrak{P}} < 1$ for every $\alpha \in K_{\mathfrak{P}}$;
    \item $|s(\alpha)|_{v} \le 1$ for every non-archimedean place $|\cdot|_v$ of $K$ such that $v \neq \mathfrak{P}$;
    \item $s(0) = 0$;
    \item $s(\alpha) = s(\beta)$ if $|\alpha - \beta|_{\mathfrak{P}} < 1$
\end{enumerate}
where $|\cdot|_{\mathfrak{P}}$ denotes the absolute value attached to the ideal ${\mathfrak{P}}$.} In particular, $s$ takes values which are integral \emph{outside}~$\mathfrak{P}$.

Each type determines an algorithm of the form~\eqref{eqn:CFalg}, which produces a $\mathfrak{P}$-adically convergent continued fraction for every input $\alpha\in K_{\mathfrak{P}}$. Obviously,  if the algorithm stops then $\alpha\in K$, but the converse is not necessarily true. One says that $K$ has the $\mathfrak{P}$-adic $\CFF$ (\emph{Continued Fraction Finiteness}) property when there exists a type $\tau=(K,\mathfrak{P},s)$ such that the corresponding algorithm stops for any given input $\alpha\in K$.

 In \cite[Theorem 4.5]{CapuanoMurruTerracini2022} the authors give a criterion for a type $\tau$ to satisfy the  $\mathfrak{P}$-adic CFF property.  They use this to show that, for instance, when $K$ has Euclidean minimum strictly less than $1$, then $K$ has the $\mathfrak{P}$-adic $\CFF$  property for almost all primes $\mathfrak{P}$ and, more generally, for all primes belonging to a norm Euclidean class of ideals (see \cite[Theorem 5.6 and Theorem 7.4]{CapuanoMurruTerracini2022}).
 
 The criterion in \cite{CapuanoMurruTerracini2022} involves the boundedness of some quantity $\nu_{\tau}$ attached to the field $K$ and to the chosen type. This is, in general, not easy to compute, but types $\tau$ with \emph{small} $\nu_{\tau}$ can be found under certain arithmetical conditions on $K$. On the other hand, it is known that the validity of the $\mathfrak{P}$-adic CFF property for a field $K$ implies certain conditions on the class group of $K$ (see Corollary \ref{cor: condizionesuffperCFF}). 
 
 In order to establish a generalized version of the $\mathfrak{P}$-adic CFF property for every number field, one needs to enlarge the set of values taken by the floor function, allowing new \emph{extraneous denominators}. This will be our framework.
 
We first generalize the definitions of floor functions and types
given in \cite{CapuanoMurruTerracini2022}  in order to include this new set of denominators $\calT$ (see Definitions \ref{def-btff} and \ref{def-type}): this gives rise to the new notions of $(\mathfrak{P},\mathcal{T})$-CFF properties for types and fields (see Definition~\ref{CFF-type-fields}). 

We now describe the main contributions of our article.  Firstly, in Theorem \ref{teo:FCcondition}, we establish a criterion for a type to satisfy the  $(\mathfrak{P},\mathcal{T})$-CFF property, which is a generalization of \cite[Theorem 4.5]{CapuanoMurruTerracini2022}. In particular we prove that a certain bound on some quantity $\nu_{\tau}$ (depending on the type $\tau$) gives a bound on the Weil height of the complete quotients appearing in the CF-expansion, which in turn implies either finiteness or periodicity. This criterion serves then as the key ingredient to prove our main result, which is the following:
\begin{theorem}
Let $K$ be a number field. There exists a finite set $\calT$ such that the field $K$ satisfies the $(\mathfrak{P},\calT)$-adic $\CFF$-property for all prime ideals $\mathfrak{P}$ outside a finite set $\mathcal{P}_0$. 
Moreover, the sets $\calT$ and $\mathcal{P}_0$ can be explicitly determined in terms of $K$.
\end{theorem}

This result is an immediate consequence of Theorem \ref{teo:fullresult-explicit}, where the sets $\calT$ and $\mathcal{P}_0$ are explicitly determined and their cardinality depends in particular on the degree and the discriminant of the field $K$ and on the covering radius of the image of the group of units $\mathcal O_K^{\times}$ via the logarithmic embedding. Using this, we are able to compute the involved constants in some explicit cases: one of them is the field $K=\QQ(\sqrt{14})$, i.e.\ the real quadratic field with the smallest discriminant which is Euclidean but not norm Euclidean. For $\QQ(\sqrt{14})$ we also show how it is possible to decrease the size of the set of denominators $\calT$ needed to ensure the finiteness result, at the cost of a larger set $\mathcal{P}_0$. This strategy could also be applied to larger classes of fields; we provide a table of examples of non-CM fields with degree $\ge 3$ over $\QQ$ and having rank of unity strictly larger than $1$ for which we compute the explicit constants of Theorem \ref{teo:fullresult-explicit}.

In the final part of the paper we compare the floor-function-based approach to continued fraction expansions with a different one, based on division chains. It is easy to see that, if a pair $a,b \in K$ has a terminating division chain of length $k$ with coefficients in a certain subring $R$ of $K$, then $a/b$ has a continued fraction expansion with partial quotients in $R$ of the same length, even if it does not come from any floor function $s:K \rightarrow K$. The existence of a finite division chain with coefficients in $R$ for every pair of elements of $K$ is also related to the fact that the group $\mathrm{SL}_2(R)$ can be generated by elementary matrices (see \cite{OMeara1965, Cohn1966, CossuZanardoZannieri2018}).
In this setting, a result of Morgan, Rapinchuck and Sury~\cite{morganEtAl:boundedGenSL2} ensures that, if $R=\calo_S$ (i.e.\ the ring of $S$-integers of $K$, where $S$ is a finite set of places) has infinitely many units, then $a/b$ has a CF-expansion of length at most $5$. 
Although this approach leads to a uniform bound on the lengths of the CF-expansions of every element in a fixed number field $K$, to the best of our knowledge there is no efficient algorithm to compute the division chains. Moreover, the continued fractions obtained by this approach do not satisfy the convergence conditions ensured by the floor-function-based approach. For these reasons we believe that floor-function-based continued fractions are worth considering.
\medskip

The paper is organized as follows. After recalling the classical properties of heights, embeddings of number fields, lattices and covering radii, in Section~\ref{sec:tFloor} we introduce the new notion of $\mathfrak P$-\textit{adic} $\calT$-\textit{floor function} for $K$, which is a natural generalization of the $\mathfrak P$-\textit{adic floor function} from~\cite{CapuanoMurruTerracini2022} -- obtained by introducing more denominators -- and the associated notions of type and CFF and CFP-properties. Section~\ref{sec:criterion} is devoted to proving a criterion for a type to satisfy the $(\mathfrak P, \calT)$-CFF and CFP properties (see Theorem \ref{teo:FCcondition}). This criterion, together with certain explicit bounds involving the covering radius of the image of the group of units via the logarithmic embedding, allows us to prove Theorem \ref{teo:fullresult-explicit}, which provides an explicit way of constructing sets $\calT$ for which the $(\mathfrak P, \calT)$-adic CFF property holds for all ideals of sufficiently large norm. In Section~\ref{sec:examples} we give some examples of computations of the constants involved in Theorem \ref{teo:fullresult-explicit} and we describe a possible strategy to decrease the size of the set of denominators used in the CF-expansions. Finally, in Section~\ref{sec:related} we analyse the relation between continued fractions and division chains, making a comparison between floor-function-based continued fractions and those obtained using division chains.

\section{Notation and preliminaries on valuations and heights}
\subsection{Absolute values, valuations and the Weil height}
In this article all number fields are assumed to be subfields of a fixed, once and for all, algebraic closure $\overline{\mathbb{Q}}$ of $\mathbb{Q}$.
Given a number field $K$, we denote by $\calo_K$ its ring of integers and by $\mathrm{Cl}(K)$ the ideal class group of $K$. We also denote by $\calM_K$ the set of places of $K$ and by $\calM_K^0$ the subset of the non-Archimedean ones.

Given a place $w\in \calM_K$, we denote by $K_w$ the completion of $K$ with respect to the $w$-adic valuation and  by $|\cdot|_w$ the corresponding absolute value. 
If $w\in \calM_K^0$ and $\mathfrak{P}\subset \calo_K$ is the corresponding prime ideal, we  usually write $K_{\mathfrak{P}}$ instead of $K_w$ and we denote by $\calo_{\mathfrak{P}}$ its ring of integers. 
With a slight abuse of notation, we shall also denote by ${\mathfrak{P}}$ the extended ideal in $\calo_{\mathfrak{P}}$. If $K=\mathbb{Q}$, then $|\cdot|_v$ is chosen so that $|p|_v=1/p$ if $v\in \calM_{\mathbb{Q}}^0$ corresponds to the prime number $p$, while $|\cdot|_v$ is either the (usual) real or complex absolute value if $v\in \calM_{\mathbb{Q}}\setminus \calM_{\mathbb{Q}}^0$. 

For  a general number field $K$ and $w\in\calM_K$, the normalization of $|\cdot|_w$ is chosen so that, for every $x\in K^\times$, the product formula \[
\prod_{w\in \calM_K}|x|_w^{d_w}=1
\]
holds, where $d_w=[K_w:\mathbb{Q}_v]$ is the local degree of $K$ at $w$ and $v\in\calM_{\mathbb{Q}}$ is such that $w\mid v$. More precisely, given $v\in \calM_\mathbb{Q}$, we have that $|x|_w=|N^{K_w}_{\mathbb{Q}_v}(x)|_v^{1/d_w}$ is the unique extension of $|\cdot|_v$ to $K_w$.

Recall that, for $x\in \overline{\mathbb{Q}}$, its absolute Weil height is the non-negative real \[
H(x)=\prod_{w\in \calM_K} \sup(1,|x|_w),
\] where $K$ is any number field containing $x$. This function satisfies several important properties as the following proposition shows.
\begin{proposition}\label{prop-H} The function $H$ satisfies the following:
\begin{enumerate}[label=(\roman*)]
\item\label{KT}
{\bf (Kronecker's theorem)} $H(x)=1$ if and only if $x$ is either 0 or a root of unity in $\overline{\mathbb{Q}}$.
\item\label{NT}
{\bf (Northcott's theorem)} For any integer $d\geq 1$ and any real $B\geq 0$ the set of all algebraic numbers of degree at most $d$ and height at most $B$ is finite and can be effectively determined.
\end{enumerate}
\end{proposition}
We refer to \cite[Theorems 1.5.9 and 1.6.8]{BombieriGubler2006} for a proof of these results and for more details on height functions.

\subsection{Embeddings}\label{sec-embeddings}
 We recall some notation and definitions from \cite[\S 5]{CapuanoMurruTerracini2022}. Given a number field $K$, we denote its signature by $(r_1,r_2)$ and its degree by $d=r_1+2 r_2$. Then $K$ has $r_1$ real embeddings $\sigma_1,\ldots,\sigma_{r_1}$ and $r_2$ pairs of complex conjugated embeddings $(\tau_1,\bar{\tau}_1),\ldots,(\tau_{r_2},\bar{\tau}_{r_2})$. We let  \[\iota\colon K\hookrightarrow \mathbb{R}^{r_1}\times \mathbb{C}^{r_2}\] be the \emph{Euclidean  embedding} defined as \[\iota(x)=(\sigma_1(x),\ldots,\sigma_{r_1}(x),\tau_1(x),\ldots,\tau_{r_2}(x)).\] 

We will also consider the embedding $\lambda \colon K\hookrightarrow \RR^d$ obtained by composing $\iota$ with the isomorphism of real vector spaces 
\begin{align*}
    f\colon \RR^{r_1}\times \CC^{r_2}&\rightarrow \RR^d\\
    (x_1,\ldots,x_{r_1},y_1,\ldots,y_{r_2}) & \mapsto (x_1,\ldots,x_{r_1},y_1,\bar{y}_1,\ldots,y_{r_2},\bar{y}_{r_2}).
\end{align*}
 
  We let $N:\RR^{r_1}\times \CC^{r_2}\rightarrow \RR$ be the \emph{norm map} defined as \[N(x_1,\ldots,x_{r_1},z_1,\ldots,z_{r_2})=x_1\cdots x_{r_1}z_1\bar{z}_1\cdots z_{r_2}\bar{z}_{r_2}\] where $\bar{z}$ denotes as usual the complex conjugate of a complex number $z$.
Notice that $N$ is the extension to $\RR^{r_1}\times \CC^{r_2}$ of the field norm on $K$, i.e.\ $N(\iota(x))=N_{K/\mathbb{Q}}(x)$.

 We also denote  by \[\ell\colon K^\times \hookrightarrow \RR^{r_1+r_2}\] the \emph{logarithmic embedding} defined as \[\ell(x)=(\log|\sigma_1(x)|,\ldots,\log|\sigma_{r_1}(x)|,2\log|\tau_1(x)|,\ldots,2\log|\tau_{r_2}(x)|)\] for every $x\in K^\times$, where  $|z|=(z\bar{z})^{1/2}$ denotes the usual complex absolute value of a complex number $z$.

 \subsection{Lattices and covering radius}\label{ss:lattices}
We further need to recall some terminology from lattice theory. 
Let $\Lambda$ be a lattice in $\RR^n$ and, for a real number $p\in [1,\infty)\cup\{\infty\}$, let $|| \cdot||_p$ be the $L_p$ norm on $\RR^n$. The \emph{distance function} relatively to $p$ is by definition
    $$\rho_p(\vv,\Lambda)=\min_{\ww\in\Lambda} ||\vv-\ww||_p.$$
    The \emph{covering radius} of $\Lambda$ with respect to $|| \cdot||_p$ is the number 
    $$\rho_p(\Lambda)=\sup_{\vv\in \mathrm{Span}(\Lambda)} \rho_p(\vv,\Lambda).$$
    Equivalently, it is the smallest number $\rho$, such
that closed balls of radius $\rho$ (with respect to the $L_p$ norm) centered on all lattice points in $\Lambda$ cover the entire space $\mathrm{Span}(\Lambda)$. 

\begin{remark}[Naive upper bound for the covering radius]
\label{rem:naiveCoveringRadius}
    Let $\Lambda \subseteq \RR^n$ be a lattice of rank $r$ with basis $\ee_1,\ldots,\ee_r$, with $\ee_i=(e_{i,1},\ldots, e_{i,n})$. Then, the set
    $$\mathcal{D}=\left \{\sum_{i=1}^r a_i\ee_i\ |\ |a_i|\leq \frac 1 2\right\} $$
    is a translate of the fundamental parallelepiped of $\Lambda$.  Therefore 
    $$\rho_\infty(\Lambda)\leq \sup_{x\in \mathcal{D}}\left(||x||_\infty\right) \leq \frac 1 2 \sum_{i=1}^r ||\ee_i||_\infty= \frac 1 2 \sum_{i=1}^r\max_{j} |e_{i,j}|, $$
    where $|\cdot|$ is the usual real absolute value.
\end{remark}

     Given a number field $K$, we define $\Lambda_K=\ell(\calo_K^\times)$ to be the image of the group of units in $\calo_K$ via the logarithmic embedding. Then, $\Lambda_K$ is a lattice in $\RR^{r_1+r_2}$ and, for  $p\in [1,\infty)\cup\{\infty\}$,  we denote its covering radius by $\rho_p(K)=\rho_p(\Lambda_K)$. 

Notice that $\mathrm{Span}(\Lambda_K)$ is the hyperplane defined by  
\[\{(x_1,\ldots,x_{r_1},y_1,\ldots,y_{r_2})\in \RR^{r_1+r_2}\mid x_1+...+x_{r_1}+y_1+...+y_{r_2}=0\}.\]

\section{\texorpdfstring{$\mathcal{T}$}{T}-floor functions and types}\label{sec:tFloor}
Given a finite set $S$ of places in $K$, the \emph{ring of $S$-integers} is
\begin{equation*}
    \mathcal{O}_{S}=\{\alpha \in K \mid v_{\mathfrak{Q}}(\alpha)\geq 0\; \text{for each}\; \mathfrak{Q} \notin S\}.
\end{equation*}
By definition, $\mathcal{O}_{S}$ consists of the elements of $K$ which are integers \emph{outside} the places in $S$. We will be particularly interested in the case $S=\{\mathfrak{P}\}$ where $\mathfrak{P}$ is a prime over $p$.
\par The following definition is modeled upon \cite[Definition 3.1] 
{CapuanoMurruTerracini2022} and provides a generalization of it.
\begin{definition}\label{def-btff}
Let $K$ be a number field, $\calT$ a finite subset of $\calo_K\setminus\{0\}$ and $\mathfrak{P}$ a prime ideal of $\calo_K$.
A  \emph{$\mathfrak{P}$-adic $\calT$-floor function for $K$} is any function $s \colon K_{\mathfrak{P}} \rightarrow K$ such that
\begin{enumerate}[label=(\roman*)]
    \item\label{item-i} $\alpha - s(\alpha)\in\mathfrak{P}$ for every $\alpha \in K_{\mathfrak{P}}$;
\item\label{item-ii} for every $\alpha\in K_\mathfrak{P}$, there exists $t\in \calT$ such that $ts(\alpha)\in \calo_{\{\mathfrak{P}\}}$;
\item\label{item-iii} $s(0)=0$;
    \item\label{item-iv} if $\alpha - \beta \in \mathfrak{P}$, then $s(\alpha) = s(\beta)$.
\end{enumerate}
\end{definition}
{The only difference with \cite[Definition~3.1]{CapuanoMurruTerracini2022} lies in condition~(ii). In \cite{CapuanoMurruTerracini2022}, this condition was stated as $s(\alpha)\in \calo_{\{\mathfrak{P}\}}$; this is precisely the special case of our formulation obtained by taking $\calT = \{1\}$.
}

As a natural generalization of \cite{CapuanoMurruTerracini2022} we also have:
\begin{definition}\label{def-type}
    A \emph{type} is a quadruple $\tau=(K,\mathfrak{P}, \calT,s)$ where $K$ is a number field, $\calT\subset \calo_K\setminus\{0\}$ is a finite subset, $\mathfrak{P}\subset \calo_K$ is a prime ideal and $s:K_{\mathfrak{P}} \rightarrow K$ is a $\mathfrak{P}$-adic $\calT$-floor function for $K$.
\end{definition}

\subsection{Continued fraction expansion associated with a type}
In this short section we recall some definitions and results from \cite[\S 3.3]{CapuanoMurruTerracini2022} on continued fraction expansions associated to types. 

Let $\tau=(K,\mathfrak{P},\calT, s)$ be a type. We denote by $w_0$ the place of $K$ corresponding to the prime $\mathfrak{P}$ and by $|\cdot|_{w_0}$ the corresponding absolute value.

As in \cite[Definition 3.5]{CapuanoMurruTerracini2022}, for a given type $\tau=(K,\mathfrak{P},\calT, s)$ we say that a \emph{continued fraction of type $\tau$} is a
(possibly infinite) sequence
$[a_0, a_1, \ldots]$ of elements of $s(K_{\mathfrak{P}})$ such that   $|a_n|_{w_0}> 1$ for every $n\geq 1$.
\begin{remark}\label{prop-an-Vn}
As remarked in \cite[\S 3.3]{CapuanoMurruTerracini2022} the following facts hold true:
\begin{enumerate}
    \item The sequence of \emph{$n$-th convergents}
    \[Q_n=
{{a_0 + \cfrac{1}{{a_1 +  \cfrac{1}{{\ddots + \cfrac{1}{{a_{n}}}}}}}}}
\] of a continued fraction $[a_0,a_1,\ldots]$ converges $\mathfrak{P}$-adically.
\item\label{def-alphan} Conversely, every $\alpha\in K_{\mathfrak{P}}$ is the limit of the sequence of convergents of the continued fraction $[a_0,a_1,\ldots]$ given by $a_n=s(\alpha_n)$, where the sequence $(\alpha_n)_n$ is defined recursively as $\alpha_0=\alpha$ and  $\alpha_{n+1}=1/(\alpha_n-s(\alpha_n))$ if $\alpha_n\neq s(\alpha_n)$ (otherwise the algorithm stops at index $n$). {For a fixed type 
$\tau$, the sequence $(a_0,a_1,\ldots)$ obtained in this way is uniquely determined. We call $[a_0,a_1,\ldots]$ the \emph{continued fraction expansion of type~$\tau$ for~$\alpha$}.}
\item\label{def_Vn} Given the continued fraction expansion $[a_0,a_1,\ldots]$ of type $\tau$ for $\alpha$, for $n\geq -1$ consider the sequence $(V_n)_n$ of elements of $K$ defined  as $V_{-1}=1$, $V_{0}=a_0-\alpha$ and $V_{n+1}=a_{n+1}V_n+V_{n-1}$. The sequence $(V_n)_n$ satisfies many interesting properties, collected for instance in \cite[Proposition 3.6]{CapuanoMurruTerracini2022}. 
 \end{enumerate}
\end{remark}
We also recall the following definitions (see \cite[Definition 4.1]{CapuanoMurruTerracini2022}):
\begin{definition}\label{CFF-type-fields}
Given a type $\tau=(K,\mathfrak{P},\calT, s)$, we say that \emph{$\tau$ satisfies the 
CFF (Continued Fraction Finiteness) property (}respectively, \emph{CFP (Continued Fraction Periodicity) property)} if every $\alpha\in K$ has a finite (respectively, finite or periodic) expansion of type $\tau$.

We say that the field \emph{$K$ satisfies the $(\mathfrak{P},\calT)$-adic CFF (respectively, CFP)
property}   if there is a type $\tau=(K,\mathfrak{P},\calT, s)$ satisfying the CFF (respectively, CFP) property.
\end{definition}

\section{A criterion for the CFF and CFP properties for types}\label{sec:criterion}
As in \cite[\S 4.1]{CapuanoMurruTerracini2022}, for $x\in\mathbb{C}$ we define 
\begin{equation}\label{def-theta}
\theta(x)= \frac{1}{2}
\left(|x|+\sqrt{|x|^2+4}\right)
\end{equation}
where $|\cdot|$ denotes the usual complex absolute value.
It is easy to see that 
\begin{equation}\label{disug-theta}
    |x|\leq \theta(x)\leq |x|+1.
\end{equation}
The following generalization of \cite[Theorem 4.5]{CapuanoMurruTerracini2022} provides a sufficient condition for a type $\tau$ to satisfy the CFP and CFF properties.
\begin{theorem}\label{teo:FCcondition}
Let $\tau=(K,\mathfrak{P},\calT, s)$ be a type. 
Denote by $w_0\in \mathcal{M}_K^0$  the place corresponding to the prime $\mathfrak{P}$. Let $\Sigma$  be the set of embeddings of $K$ in $\CC$  and set 
\[
\nu_\tau=\sup\left\{ {|a|_{w_0}^{-d_{w_0}}}{\prod_{\sigma\in\Sigma}\theta(\sigma(a))\prod_{w\in \mathcal{M}_K^0\setminus\{w_0\}} \max(|a|_w,1)^{d_w}}  \;\middle |\;  a\in s(K_\mathfrak{P}), |a|_{w_0}>1\right\}. 
\]
Then
\begin{enumerate}
    \item[(a)]\label{point-a} if $\nu_{\tau}\leq 1$, then $\tau$ satisfies the CFP property;
    \item[(b)]\label{point-b} if $\nu_{\tau}< 1$, then $\tau$ satisfies the CFF property. 
\end{enumerate}
More precisely, in both cases, for every $\alpha\in K$ the continued fraction expansion of type $\tau$ of $\alpha$ has either length or period length bounded by \[C_{\alpha}=d\left(2^{d+1}\left\lceil\sqrt{
|s(\alpha)-\alpha|^2+1}\prod_{w\in \calM_K^0\setminus \{w_0\}}\sup(|s(\alpha)-\alpha|_w,1) \right\rceil +1\right)^{d+1},
\] 
where $d=[K:\mathbb{Q}]$.
\end{theorem}
\begin{remark} Before proving the result we notice that, if  $a\in s(K_\mathfrak{P})$ then, from Definition \ref{def-btff}.\ref{item-ii}, there exists $t\in \calT$ such that $|ta|_w\leq 1$  for every $w\in {M}_K^0\setminus\{w_0\}$. So the product 
$\prod_{w\in \mathcal{M}_K^0\setminus\{w_0\}}\max(|a|_w,1)$ is finite and bounded, independently of $a$, by the quantity $\prod_{t\in \calT}\prod_{w\in \mathcal{M}_K^0\setminus\{w_0\}}|t|_w^{-1}$.
\end{remark}
\begin{proof}
Let $\alpha\in K$, let $[a_0,a_1,\ldots]$ be its continued fraction expansion of type $\tau$ and  let $(\alpha_n)_n$ be the sequence defined in Remark \ref{prop-an-Vn}\eqref{def-alphan}. 
Let $H$ be the absolute Weil height on $\overline{\mathbb{Q}}$. We first want to show that, for every $n\geq 1$, one has \begin{equation}\label{boun-H-to-show}H(\alpha_{n+1})^d \leq C\,\nu_\tau^n,
\end{equation}
where $d=[K:\mathbb{Q}]$ and $C>0$ is a constant only depending on $\alpha$.\\
If $(V_n)_n$ is the sequence defined in Remark \ref{prop-an-Vn}\eqref{def_Vn}, then   $\alpha_{n+1}=-V_{n-1}/V_n$ for all $n\geq 1$ by~\cite[Proposition 3.6(b)]{CapuanoMurruTerracini2022}, and since $V_n\in K$ we have
\begin{equation}\label{H-bound}
    H(\alpha_{n+1})^d  = 
    H\left(-\frac{V_{n-1}}{V_n}\right)^d
    = \prod_{\mathclap{w\in \calM_K}}\sup\left( \left|\frac{V_{n-1}}{V_n}\right |_w^{d_w},1\right)
    =\prod_{\mathclap{w\in \calM_K}}\sup(|V_n|_w^{d_w},|V_{n-1}|_w^{d_w}).
\end{equation}
Let $w$ be an Archimedean place of $K$ and let  $\sigma\in\Sigma$ be the corresponding embedding. Notice that, for each such $\sigma$, the sequence $\sigma(V_n)$ satisfies the recurrence formula
$$\sigma(V_n)=\sigma(a_n) \sigma(V_{n-1}) +\sigma(V_{n-2}),$$
for every $n\ge 1$. Therefore, by \cite[Lemma 4.4]{CapuanoMurruTerracini2022} we have
\begin{equation} \label{eq:theta_est}
\prod_{w\in \calM_K\setminus\calM_K^0}\sup(|V_n|_w^{d_w},|V_{n-1}|_w^{d_w})=\prod_{\sigma\in\Sigma}\sup(|\sigma(V_n)|, |\sigma(V_{n-1})| )\leq C_\infty \prod_{j=1}^n \prod_{\sigma\in \Sigma} \theta(\sigma(a_j))
\end{equation}
where $C_\infty=\sqrt{
|a_0-\alpha|^2+1}$.

Let now $w\in \calM_K^{0}\setminus \{w_0\}$. Then, since for all $n\geq 1$ one has
\[\sup(|V_n|_w,|V_{n-1}|_w) \leq \sup(|V_{n-1}|_w,|V_{n-2}|_w)\sup(|a_n|_w,1),\] 
iterating one gets
\begin{equation}\label{w-non-arch}
\sup(|V_n|_w,|V_{n-1}|_w)\leq C_w\prod_{j=1}^n \sup (|a_j|_v,1)
\end{equation}
where  $C_w=\sup(|w_0|_{w}, |V_1|_w)=\sup(|a_0-\alpha|_w,1)$ and clearly $C_w=1$ for almost all~$w$. 

Finally, if $w=w_0$, {from \cite[Proposition 3.6, points (c) and (g)]{CapuanoMurruTerracini2022} one has that for every $n\geq 1$, \[|V_n|_{w_0}=\prod_{i=1}^{n+1}\frac{1}{|a_i|_{w_0}},\] and $|a_n|_{w_0}\geq 1$. Hence }\begin{equation}\label{w=w_0}\sup(|V_n|_{w_0}^{d_{w_0}},|V_{n-1}|_{w_0}^{d_{w_0}})=|V_{n-1}|_{w_0}^{d_{w_0}}.\end{equation}

But now, by \eqref{H-bound} and \eqref{w=w_0} we have
\[
        H(\alpha_{n+1})^d = |V_{n-1}|_{w_0}^{d_{w_0}}  \prod_{\sigma \in \Sigma} \max(|\sigma(V_n)|,|\sigma(V_{n-1})| )\prod_{w\in \calM_K^0\setminus \{w_0\}}\max(|V _n|_w,|V_{n-1}|_w)^{d_w}.\]
        Therefore, using \eqref{eq:theta_est} and \eqref{w-non-arch} and rearranging the factors we get  
    \begin{align*}
    H(\alpha_{n+1})^d
    &\leq C \cdot |V_{n-1}|_{w_0}^{d_{w_0}} \prod_{j=1}^n \left(\prod_{\sigma\in \Sigma} \theta(\sigma(a_j))\prod_{w\in \calM_K^0\setminus \{w_0\}} \sup(|a_j|_w,1)^{d_w}\right) \\
    &\leq C\cdot |V_{n-1}|_{w_0}^{d_{w_0}} \nu_\tau^n \prod_{j=1}^n|a_j|_{w_0}^{d_{w_0}},
   \end{align*}
   where 
   \[
   C=C_{\infty}\cdot\prod_{w\in \calM_K^0\setminus \{w_0\}} C_w 
   \] 
   is an explicit constant depending only on $\alpha$.
   Finally, by \cite[Proposition 3.6.(d),(g)]{CapuanoMurruTerracini2022},
    we have \[|V_{n-1}|_{w_0}=\left|(-1)^n\prod_{j=1}^{n} \frac 1{\alpha_j}\right|_{w_0}=\prod_{j=1}^{n} \left |\frac 1{a_j}\right |_{w_0} \] so that
    $$ H(\alpha_{n+1})^d \leq C\,\nu_\tau^n$$ and \eqref{boun-H-to-show} is proven, as desired.
    
   Suppose now that $\nu_{\tau}\leq 1$. In view of \eqref{boun-H-to-show}, for every $n\geq 1$ \[\alpha_n\in \mathcal{S}_{d,C}:=\{\alpha\in \overline{\mathbb{Q}}\mid [\mathbb{Q}(\alpha):\mathbb{Q}]\leq d, H(\alpha)\leq C^{1/d}\}.\] 
   The set $\mathcal{S}_{d,C}$ is finite by Northcott's theorem and, by \cite[proof of Theorem 1.6.8]{BombieriGubler2006} its cardinality can be bounded as \[|\mathcal{S}_{d,C}|\leq C_{\alpha}:=d(2^{d+1}\lceil C\rceil+1)^{d+1}.\]  
   
   Therefore either the continued fraction expansion of type $\tau$ for $\alpha$ is finite, of length at most $C_{\alpha}$, or there exist $m,n\in\NN$ such that $\alpha_m=\alpha_n$, so that the expansion is periodic of period length at most $C_{\alpha}$. 
   If moreover  $\nu_{\tau}< 1$, then $H(\alpha_{n+1})<H(\alpha_n)$ for all $n$ and, by Northcott's theorem, this can happen only for finitely many $n$, so the expansion of $\alpha$ must be finite.
    \end{proof}

\section{The \texorpdfstring{$\CFF$}{CFF} property for fields}
\label{sec:CFF}
 \subsection{A bound involving the covering radius}
 
We have the following effective version of Lemma 5.5 of \cite{CapuanoMurruTerracini2022}.
\begin{lemma}\label{lem:effective-5.5} Let $K$ be a number field, let $\Sigma$  be the set of embeddings of $K$ in $\CC$ and denote by $\calo_K^{\times}$ the units in $\calo_K$. 
Set $T_0=e^{\rho_\infty(K)} $, where $\rho_\infty(K)$ denotes the covering radius as in Section \ref{ss:lattices}. Then, for every $a\in K^\times $, there exists $u\in \calo_K^\times$ such that
$$ |\sigma(ua)|  \leq T_0 \sqrt[d]{|N_{K/\QQ}(a)|}$$
 for every $\sigma\in\Sigma $, where $|\cdot|$ is the usual complex absolute value.
\end{lemma}
\begin{proof} 
For $a\in K^\times$, consider the vector $$\mathbf{b}=\ell(a)-\frac {\log{|N_{K/\QQ}(a)|}} d (1,\ldots,1,2,\ldots,2) $$ where $d=[K:\mathbb{Q}]$.

By construction, $\mathbf{b}\in \mathrm{Span}(\Lambda_K)$. Hence, by definition of the covering radius, there exists $u\in \calo_K^\times$ such that $||\mathbf{b}+\ell(u)||_\infty\leq \rho_\infty(K)$.  But now, denoting by $\sigma_1,\ldots,\sigma_{r_1}\in \Sigma$ the real embeddings of $K$ and by  $(\tau_1,\bar{\tau}_1),\ldots,(\tau_{r_2},\bar{\tau}_{r_2})$ the $r_2$ pairs of complex conjugated embeddings, we have that \[\mathbf{b}+\ell(u)=(v_1,\ldots,v_{r_1},w_1,\ldots,w_{r_2})\] where, for all $1\leq i\leq r_1$ and $1\leq j\leq r_2$, one has  
\[v_i=\log \frac {|\sigma_i(a)|}{\sqrt[d]{|N_{K/\QQ}(a)|}} +\log(|\sigma_i(u)|)= \log \frac {|\sigma_i(au)|}{\sqrt[d]{|N_{K/\QQ}(a)|}}\]

and \[w_j=2\log \frac {|\tau_j(a)|}{\sqrt[d]{|N_{K/\QQ}(a)|}} +2\log(|\tau_j(u)|)=2 \log \frac {|\tau_j(au)|}{\sqrt[d]{|N_{K/\QQ}(a)|}}.\]

Therefore \[\rho_\infty(K)\geq ||\mathbf{b}+\ell(u)||_\infty= \max(|v_1|,\ldots,|v_{r_1}|,|w_1|,\ldots,|w_{r_2}|)\geq \left| \log \frac {|\sigma(au)|}{\sqrt[d]{|N_{K/\QQ}(a)|}}\right|\] for every $\sigma\in \Sigma$.
Thus, for every $\sigma\in \Sigma$, one has 
$\left |{\sigma(au)}\right|\leq T_0 \sqrt[d]{|N_{K/\QQ}(a)|}$, where $T_0=e^{\rho_\infty(K)}$, as wanted.
\end{proof}
\subsection{A criterion for \texorpdfstring{$(\mathfrak{P},\calT)$}{P,T}-adic \texorpdfstring{$\CFF$}{CFF}-property for a field}
Let $K$ be a number field of degree $d$ and let 
$\iota:K\hookrightarrow \RR^{r_1}\times \CC^{r_2}$ and $N:\RR^{r_1}\times\CC^{r_2}\rightarrow \RR$ be, respectively, be the Euclidean embedding and the norm map defined in Section \ref{sec-embeddings}. For an ideal $\mathfrak{I}\subset \calo_K$, we denote by $\mathcal{N}(\mathfrak{I})=|\calo_K/\mathfrak{I}|$ its (ideal) norm while, for an element $x\in K$, we denote by $N_{K/\QQ}(x)$ its (field) norm.

\begin{definition}Let $\mathfrak{A}$ be a non-zero ideal of $\mathcal{O}_K$, let $\mathcal{T}\subseteq \calo_K \setminus\{0\}$ be a finite subset and let $0<\delta < 1$. We say that \emph{property ($\star$) holds for the triple $(\mathfrak{A},\calT,\delta)$} if, for each  $\xi\in \RR^{r_1}\times \CC^{r_2}$, there exists $\omega\in\iota(\mathfrak{A})$ and  $x\in\iota(\calT)$   such that  $|N(x\xi -\omega)| <\delta \mathcal{N}(\mathfrak{A})$.
\end{definition}
Notice that property ($\star$)  holds for the triple $(\mathfrak{A},\calT,\delta)$ if and only if it holds for any triple $(\mathfrak{B},\calT,\delta)$, where $\mathfrak{B}$ is any element of the ideal class of $\mathfrak{A}$. Therefore we have the following definition:
\begin{definition}\label{eq:Dproperty} 
    Let $\calC$ be an ideal class in $\calo_K$, let $\mathcal{T}\subseteq \calo_K $ be a finite subset and let $0<\delta \leq 1$. We say that \emph{property ($\star$) holds for the triple $(\calC,\calT,\delta)$} if one of the following equivalent conditions is satisfied:
    \begin{itemize}
    \item[(i)] property ($\star$) holds for one triple $(\mathfrak{A},\calT,\delta)$ with $\mathfrak{A}\in \mathcal{C}$;
    \item[(ii)] property ($\star$) holds for all triples $(\mathfrak{A},\calT,\delta)$ with $\mathfrak{A}\in \mathcal{C}$.
    \end{itemize}
\end{definition}

\begin{theorem}\label{teo:fullresult}
Let $K$ be a number field, $\calT\subset\calo_K\setminus\{0\}$ be a finite set, $\mathcal{C}$ an ideal class of $\calo_K$ and $0<\delta <1$. Suppose that property  ($\star$) holds for the triple $(\mathcal{C},\mathcal{T}, \delta)$.
Then $K$ satisfies the $(\mathfrak{P},\calT)$-adic $\CFF$-property for all but finitely many prime ideals $\mathfrak{P}$ of $\mathcal{C}$. 
\end{theorem}

\begin{proof} 
As usual, we denote by $||\cdot||_{\infty}$ the sup norm of a vector and by $|\cdot|$ the usual complex absolute value.

We fix an integral ideal $\mathfrak{A}\in\calC$ and we denote by $\calD\subseteq \RR^{r_1}\times \CC^{r_2}$ a compact fundamental domain for $\iota(\mathfrak{A})$. 
From the fact that Property ($\star$) holds for $(\mathcal{C},\mathcal{T}, \delta)$, that $\calD$ is compact and that $N$ is continuous, we deduce that  there exists a finite number of pairs $(x_1,\gamma_1),\ldots,(x_t,\gamma_t)$ with $x_i\in \iota(\calT)$, $\gamma_i\in\iota(\mathfrak{A})$  such that the open sets \[U_i=\{\alpha\in\calD \mid |N(x_i\alpha+\gamma_i)|<\delta \mathcal{N}(\mathfrak{A}) \} \] 
form a finite covering of $\calD$. 

Moreover, for every $\alpha\in\calD$, we also have that $|| x_i\alpha+\gamma_i||_\infty < \Gamma$, where \[\Gamma=\max_i||x_i||_{\infty}\max_{\alpha\in \mathcal{D}}||\alpha||_{\infty}+\max_i||\gamma_i||_\infty\] is a well-defined positive constant only depending on the finitely many pairs $(x_i,\gamma_i)$ and on $\mathfrak{A}$.

 Let now $\mathfrak{P}$ be a prime ideal in $\calC$ and suppose that $\mathfrak{P}\cap\mathcal{T}=\emptyset$. Since $\mathcal{T}$ is finite, this will be true for all but finitely many $\mathfrak{P}\in\mathcal{C}$. We want to show that, for all but finitely many of such ideals $\mathfrak{P}$, there is a type $(K,\mathfrak{P},\calT, s)$  satisfying the CFF property.\\

To this aim, we now construct a suitable $\mathfrak{P}$-adic $\calT$-floor function $s \colon K_{\mathfrak{P}} \rightarrow K$. Let $\eta\in K_{\mathfrak{P}}$ and suppose that $\eta$ belongs to the coset $\alpha+\mathfrak{P}\calo_{\mathfrak{P}} \subseteq  K_{\mathfrak{P}}$.
 By strong approximation, such coset contains an element $\alpha'\in K$ such that $|\alpha'|_v\leq 1$ for every non-Archimedean $v\in \calM_K\setminus\{w_0\}$, where $w_0$ is the place of $K$ corresponding to $\mathfrak{P}$.   

Also, by Lemma \ref{lem:effective-5.5}, we can take an element $\gamma_{\mathfrak{P}}\in K$ such that $\gamma_{\mathfrak{P}}\mathfrak{A}=\mathfrak{P}$ and $|\sigma(\gamma_{\mathfrak{P}})| \leq T_{0} \sqrt[d]{|N_{K/\QQ}(\gamma_\mathfrak{P})|}$ for every embedding $\sigma$ of $K$ in $\mathbb{C}$, where $T_0=e^{\rho_{\infty}(K)}$ and $d=[K:\QQ]$.

Now, take $y\in \mathfrak{A}$ such that the translate \[\beta=\frac{\alpha'}{\gamma_{\mathfrak{P}}}+y \] of the element ${\alpha'}/{\gamma_{\mathfrak{P}}}$ satisfies $\iota(\beta)\in \calD$.
 
Furthermore, by previous considerations, there exist $x_i\in \iota(\mathcal{T})$ and $\gamma_i\in\iota(\mathfrak{A})$ such that $|N(x_i\iota(\beta)-\gamma_i)|<\delta \mathcal{N}(\mathfrak{A})$, and $||x_i\iota(\beta)-\gamma_i||_\infty < \Gamma$. Finally,  writing $x_i=\iota(t_i)$ with $t_i\in \mathcal{T}$ and $\gamma_i=\iota(a_i)$ with $a_i\in\mathfrak{A}$, we set  \[ s(\eta)=\gamma_{\mathfrak{P}}\left(\beta-\frac{a_i}{t_i}\right).\]

We now want to verify that the function $s:K_\mathfrak{P}\rightarrow K$ so defined is indeed a  $\mathfrak{P}$-adic $\calT$-floor function for $K$, as in Definition \ref{def-btff}.

First notice that condition \ref{item-iii} clearly holds, as well as condition \ref{item-iv}, as, by construction, $s(\eta)$ only depends on the class of $\eta$ modulo $\mathfrak{P}$. So we are left to verify conditions \ref{item-i} and \ref{item-ii}.

For every $\eta\in K_{\mathfrak{P}}$, we have that \[s(\eta)-\eta=\gamma_{\mathfrak{P}}y-\gamma_{\mathfrak{P}}\frac{a_i}{t_i} \mod \mathfrak{P}.\]
As both $\gamma_{\mathfrak{P}}y,\gamma_{\mathfrak{P}}a_i\in \gamma_{\mathfrak{P}}\mathfrak{A}=\mathfrak{P}$, we have $t_i(s(\eta)-\eta)\in \mathfrak{P}$. But now notice that $t_i\not\in\mathfrak{P}$ (indeed $t_i \in \calT$ and, by hypothesis, $\calT\cap\mathfrak{P}=\emptyset$). This implies that $s(\eta)-\eta\in \mathfrak{P}$, condition \ref{item-i} is satisfied.
 
As for condition \ref{item-ii}, let $\eta\in K_\mathfrak{P}$. We want to show that there exists $t\in \calT$ such that $xs(\eta)$ is $\mathfrak{Q}$-integral for every prime $\mathfrak{Q}\not= \mathfrak{P}$ and we claim that we can take $t=t_i$. Indeed, if $v$ is the place corresponding to $\mathfrak{Q}$, we have
\begin{equation}\label{|ta|_v}
\begin{split}
    |t_i s(\eta)|_v&= |t_i\gamma_{\mathfrak{P}}\beta -\gamma_{\mathfrak{P}}a_i|_v= |t_i\alpha'+t_i\gamma_{\mathfrak{P}}y- \gamma_{\mathfrak{P}}a_i|_v\\
    &\leq \max(|t_i|_v|\alpha'|_v,|t_i\gamma_{\mathfrak{P}}y- \gamma_{\mathfrak{P}}a_i|_v)\leq1
\end{split}
\end{equation}
as $|\alpha'|_v\leq 1$ and both $t_i$ and $t_i\gamma_{\mathfrak{P}}y- \gamma_{\mathfrak{P}}a_i$ are in $\calo_K$.

We want now to apply Theorem \ref{teo:FCcondition} to show that the type $(K,\mathfrak{P},\calT, s)$ associated to the floor function $s$ constructed above satisfies the $\CFF$ property for all but finitely many prime ideals $\mathfrak{P}\in\calC$. To this aim we have to show that for every $a\in s(K_{\mathfrak{P}})$ such that $|a|_{w_0}>1$ we have 
\begin{equation}\label{nu-tau-a} {|a|_{w_0}^{-d_{w_0}}}{\prod_{\sigma\in\Sigma}\theta(\sigma(a))\prod_{w\in \mathcal{M}_K^0\setminus\{w_0\}} \max(|a|_w,1)^{d_w}} <1\end{equation}
where $\theta$ was defined in \eqref{def-theta}.

Let us set $a=s(\eta)$ with $\eta\in K_{\mathfrak{P}}$ as before.
We first give a bound for the factor $\prod_{\sigma\in\Sigma}\theta(\sigma(a))$.
Notice that we have 
\[|N(x_i \iota(a))|=|N(\iota(\gamma_{\mathfrak{P}}))N(x_i\iota(\beta)-\gamma_i)|<\delta |N_{K/\mathbb{Q}}(\gamma_{\mathfrak{P}})|\mathcal{N}(\mathfrak{A})=\delta \mathcal{N}(\mathfrak{P})\] so, setting $q=\mathcal{N}(\mathfrak{P})$, we get
 \begin{equation}\label{bound-N(iota(a))}
     |N(\iota(a))|<\frac {\delta q} {|N(x_i)|}.
 \end{equation}

Also, we have $||x_i\iota(a)||_\infty =
||\iota(\gamma_{\mathfrak{P}})||_\infty||x_i\iota(\beta)-\gamma_i||_\infty < \left(T_0\sqrt[d]{|N_{K/\QQ}(\gamma_{\mathfrak{P}})|}\right)\Gamma$, so that
\begin{equation}\label{bound-||iota(a)||}
||\iota(a)||_{\infty} < \frac{T_0 {\Gamma}\sqrt[d]{|N_{K/\QQ}(\gamma_{\mathfrak{P}})|}}{||x_i||_{\infty}}.
\end{equation}
 Denoting by $\Sigma$ the embeddings of $K$ into $\mathbb{C}$, for every subset $S\subsetneq \Sigma$ we have
 \begin{equation}
   \prod_{\sigma\in S} |\sigma(a)|\leq (\Gamma_0 T_0)^{|S|} \sqrt[d]{q^{|S|}}
\end{equation}
where $\Gamma_0=\Gamma/\min_i ||x_i||_{\infty}$ depends on $\mathfrak{A}$ and $\calT$.
Therefore, using also \eqref{disug-theta}, we get 
\begin{align*} 
\prod_{\sigma\in\Sigma}\theta(\sigma(a))&\leq \prod_{\sigma\in\Sigma }(1+|\sigma(a)|)
\leq \sum_{S\subseteq \Sigma} \prod_{\sigma\in S} |\sigma(a)| =|N(\iota(a))|+ \sum_{S\subsetneq \Sigma} \prod_{\sigma\in S} |\sigma(a)|\\
&< \frac{\delta q}{|N(x_i)|} + \sum_{i=0}^{d-1} {{d}\choose{i}}(\Gamma_0 T_0\sqrt[d]{q})^i\leq \frac{\delta q}{|N(x_i)|} + \Gamma_1 \sqrt[d]{q^{d-1}}
\end{align*}
where for instance one can take \[\Gamma_1=\left(2\max(1,\Gamma_0T_0)\right)^{d-1}\]
which depends only on $K$, $\mathfrak{A}$ and $\calT$, but not on $\mathfrak{P}$. 
Hence, for all $\delta<\epsilon<1$, if \begin{equation}\label{low-bound-q}q>\left(\frac{\Gamma_1|N(x_i)|}{(\epsilon-\delta)}\right)^d\end{equation} we obtain
\begin{equation}\label{bound-prod-sigma}
\prod_{\sigma\in\Sigma}\theta(\sigma(a))<\frac{\epsilon q}{|N(x_i)|}.
\end{equation}

Let $w\not= w_0$ be a non-Archimedean place of $K$. By \eqref{|ta|_v}, we have that $|a|_w\leq \frac 1 {|t_i|_w}$ and $t_i\in\calo_K\setminus \mathfrak{P}$, so that $|t_i|_w\leq 1$
 and $|t_i|_{w_0}=1$. Putting this together with the product formula gives \begin{equation}\label{bound-factor2}
 \prod_{\mathclap{w\in\calM_K^0\setminus\{w_0\}}}\max(|a|_w,1)^{d_w} \leq \prod_{\mathclap{w\in\calM_K^0\setminus\{w_0\} }}\max\left(\frac {1}{|t_i|_w},1\right)^{d_w} = \prod_{w\in \calM_K^0} \frac {1} {|t_i|_w^{d_w}}=|N_{K/\QQ}(t_i)|.
\end{equation}

Since by hypothesis $|a|_{w_0}>1$, then $|a|_{w_0}^{d_{w_0}}\geq |a|_{w_0}\geq q$. 
This, together with \eqref{bound-prod-sigma} and \eqref{bound-factor2} (recalling also that $N(x_i)=N(\iota(t_i))=N_{K/\QQ}(t_i)$), gives that the left-hand side in \eqref{nu-tau-a} is bounded above by $\epsilon<1$ for any $q$ satisfying \eqref{low-bound-q}, so Theorem \ref{teo:FCcondition} applies and concludes the proof.
\end{proof}

\begin{remark}
Although the validity of property ($\star$) for $\calC$ is independent from the choice of the ideal $\mathfrak{A}$, it is important to highlight that the bound for the norm of the ideals $\mathfrak{P}$ satisfying the theorem does depend on the specific choice of the ideal $\mathfrak{A}$ (as well as on the set $\calT$ and the field $K$) and on the chosen covering for $\calD$.
\end{remark}
\subsection{Explicit realization of the \texorpdfstring{$(\mathfrak{P},\calT)$}{(P,T)}-adic \texorpdfstring{$\CFF$}{CFF}-property}

The goal of this section is to prove Theorem \ref{teo:fullresult-explicit}, which provides an explicit way of constructing sets $\mathcal{T}$ for which the  $(\mathfrak{P},\calT)$-adic $\CFF$-property holds on all ideals of sufficiently large norm. \par At first glance, this statement might seem stronger than Theorem \ref{teo:fullresult} from the previous section. However, while we gain an explicit lower bound for the norm of the ideals for which it holds, we lose flexibility in the choice of the set of denominators~$\calT$.

The heart of the proof of Theorem~\ref{teo:fullresult-explicit} is the following effective and slightly more general version of a result originally proven by Hurwitz \cite{Hurwitz1919} and based on its ideas. We also refer the reader to \cite[Theorem 1.4]{Lenstra1976} and \cite[Theorem 2.1]{LongThistlethwaiteMorwen2016} for a detailed description of the original proof.

\begin{theorem}\label{teo:Hurwitz3} Let $K$ be a number field with signature $(r_1,r_2)$ and let $\Delta$ be the absolute discriminant of $K$.
Let $\mathfrak{A}$ be an integral ideal in $\calo_K$ of norm $\mathcal{N}(\mathfrak{A})$ and let \begin{equation}\label{cont-C(A,K)}
c(\mathfrak{A},K)=\max\left( \left (\frac 2 \pi\right)^{r_2}\sqrt{|\Delta|}\mathcal{N}(\mathfrak{A}),1\right).
\end{equation}

Then, for every rational integer $M>c(\mathfrak{A},K)$, there exists a real number $\epsilon=\epsilon(\mathfrak{A})$ with $0<\epsilon <1$ satisfying the following property:
for each $\xi\in K$, there is $\tau\in\mathfrak{A}$ and a rational integer $0<j<M$  for which
$|\sigma(j\xi -\tau)|<\epsilon$ for every  embedding $\sigma$ of $K$ in $\CC$.
In particular we have \[|N_{K/\QQ}(j\xi -\tau)|<\epsilon^d.\]
\end{theorem}
\begin{proof} 
We consider the embedding $\lambda=f\circ \iota:K\rightarrow \RR^d$ defined in Section \ref{sec-embeddings} and we denote as usual by $||x||_{\infty}=\max_i(|x_i|)$ the sup norm of a vector $x\in \RR^{d}$, where $|\cdot|$ is the usual absolute value on $\RR$. For every $\epsilon>0$, we consider the open ball  \[U_\epsilon=\left\{x \in \RR^{d} \mid ||x||_{\infty}< \frac{\epsilon} 2\right\}.\]
Then for every $x,y\in U_{\epsilon}$ one has \begin{equation}\label{diff-x-y-eps}||x-y||_{\infty}<\epsilon.\end{equation} Moreover 
\[\mathrm{Vol}(U_\epsilon)=\epsilon^{d}\left (\frac{\pi} 4\right )^{r_2}.\]

Let $\mathcal{D}\subseteq \RR^d$ be a  compact fundamental parallelepiped for $\lambda(\mathfrak{A})$. It is known that 
\[\mathrm{Vol}(\mathcal D)= \frac{\sqrt{|\Delta|}N(\mathfrak A)}{2^{r_2}},\] (see for example \cite[\S4.2, Proposition 2]{Samuel1967}); therefore, by hypothesis, there exists $0<\epsilon<1$ such that \begin{equation}\label{choice-M}
M>\frac {\mathrm{Vol}(\mathcal D)}{\mathrm{Vol}(U_\epsilon)}.
\end{equation}
Let $\varphi_{\mathcal{D}}: \mathbb{R}^d \rightarrow \calD$  be the map sending a vector to its representative (modulo the ideal ${\mathfrak{A}}$) lying in $\mathcal{D}$.  Let $\xi\in K$ and for $r\in\{1,\ldots, M\}$ consider the set \[V_r=\varphi_{\calD}(r\lambda(\xi)+U_\epsilon).\] We now show that $\varphi_{\calD}$ is injective on $r\lambda(\xi)+U_\epsilon$. Suppose that $r\lambda(\xi)+u_1=r\lambda(\xi)+u_2+\lambda(\tau)$ for some $u_1,u_2\in U_\epsilon$ and $\tau\in \mathfrak{A}$. Then  $\lambda(\tau)\in U_\epsilon$ and  $||\lambda(\tau)||_{\infty}\leq \epsilon/2$, hence $$|N_{K/\QQ}(\tau)|\leq \left(\frac{\epsilon}{2}\right)^d<1.$$ As $\tau\in\calo_K$ we get $\tau=0$, and $u_1=u_2$ as wanted. 
 
 Therefore  $\mathrm{Vol}(V_r)=\mathrm{Vol}(\varphi_{\calD}(r\lambda(\xi)+U_\epsilon)=\mathrm{Vol}(U_\epsilon)$ for every $r$. By our choice of $\epsilon$ and \eqref{choice-M}, the $V_r$'s cannot be disjoint in pairs. Therefore there exist $1\leq r<s\leq M$ such that $V_r\cap V_s\not= \emptyset$, that is $s\lambda(\xi)+u_1+\lambda(\tau)=r\lambda(\xi)+u_2$ for some $u_1,u_2\in U_\epsilon$ and $\tau\in \mathfrak{A}$. Then \[(s-r)\lambda(\xi)+\lambda(\tau) =u_2-u_1\] so that, by \eqref{diff-x-y-eps}, for every embedding $\sigma$ of $K$ in $\CC$ we have \[|\sigma((s-r)\lambda(\xi)+\lambda(\tau))|\leq ||\lambda((s-r)\xi+\tau)||_{\infty}=||u_1-u_2||_{\infty}<\epsilon\] and the statement is proven taking $j=s-r$.
\end{proof}

\begin{corollary}\label{cor:idealeintero} Let $K$ be a number field of degree $d$ and signature $(r_1,r_2)$ and let $\Delta$ be the absolute discriminant of $K$.  Let \begin{equation}\label{C(K)}
c(K)=\max\left( |\Delta|\left (\frac 8 {\pi^2}\right)^{r_2} \frac{d!}{d^d},1\right).
\end{equation}

Then, for every rational integer $M\geq c(K)$,  there exists a real number $\epsilon$ with $0<\epsilon <1$ such that
 every ideal class $\calC$ in $\mathrm{Cl}(K)$ contains an integral ideal $\mathfrak{A}=\mathfrak{A}(\calC)$ with the following property: 
 for each $\xi\in K$, there is $\tau\in\mathfrak{A}$ and a rational integer $0<j<M$  for which
$|\sigma(j\xi -\tau)|<\epsilon$ for every  embedding $\sigma$ of $K$ in $\CC$.
In particular we have \[|N_{K/\QQ}(j\xi -\tau)|<\epsilon^d.\]
 \end{corollary}
\begin{proof}
Let $\mathrm{Cl}(K)=\{\calC_1,\ldots,\calC_m\}$ be the ideal class group of $K$.
By Minkowski's bound (see for example \cite[\S 4.3,  Corollaire 1]{Samuel1967}), every class $\calC_i$ in $\mathrm{Cl}(K)$ contains an integral ideal $\mathfrak{A}_{i}$ of norm \[\mathcal{N}(\mathfrak{A}_i) < \frac{d!}{d^d} \left (\frac 4 \pi\right)^{r_2}\sqrt{|\Delta|}.\] But then we can apply Theorem \ref{teo:Hurwitz3} as, by hypothesis, $M>c(\mathfrak{A}_i,K)$, where $c(\mathfrak{A},K)$ is the constant defined in \eqref{cont-C(A,K)}. Choosing $\epsilon(\mathfrak{A}_i)>0$ as in Theorem \ref{teo:Hurwitz3}, the proof is concluded by setting $\epsilon=\min_{i}\epsilon(\mathfrak{A}_i)$.
\end{proof}

\begin{theorem}\label{teo:fullresult-explicit} Let $K$ be a number field of degree $d$, let $\Delta$ be its the absolute discriminant and let $\rho_{\infty}(K)$ be the covering radius as in Section \ref{ss:lattices}. Let $c(K)$ be the constant defined in \eqref{C(K)}, let $M>c(K)$ be a rational integer, let $\epsilon=\epsilon(M)>0$ be as in the statement of Corollary \ref{cor:idealeintero} and, for $T_0=e^{\rho_{\infty}(K)}$, define

\begin{equation}
\label{eq:mostro}
c(M,K)= \left(\frac {M}{\epsilon T_0\left ( \sqrt[d]{\frac {1-\epsilon^d}{\epsilon^dT_0^d} +1 } -1\right )}\right)^{d}.
\end{equation}

Then, setting $\calT=\{1,\ldots, M\}$, the field $K$ satisfies the $(\mathfrak{P},\calT)$-adic $\CFF$-property for every prime ideal $\mathfrak{P}$ of $\calo_K$ of norm
$\mathcal{N}(\mathfrak{P})>c(M,K)$. 
\end{theorem}

\begin{remark}
\label{rem:monsterVSM^d}
    The constant $c(M,K)$ defined in~\eqref{eq:mostro} is strictly larger than $M^d$. To prove this, we use the facts that $0<\epsilon<1$ and $\sqrt[d]{a+b}\leq \sqrt[d]{a} + \sqrt[d]{b}$ for any $a,b\geq 0$. The following inequalities hold for the denominator in~\eqref{eq:mostro}:
    \begin{align*}
        \epsilon T_0\left( \sqrt[d]{\frac {1-\epsilon^d}{\epsilon^dT_0^d} +1 } -1\right) <  T_0\left( \sqrt[d]{\frac {1}{T_0^d} +1 } -1\right) \leq  T_0\left( \sqrt[d]{\frac {1}{T_0^d}} \right) = 1.
    \end{align*}
  In particular, the condition $\mathcal{N}(\mathfrak{P})>c(M,K)$ implies that no element in $\mathcal{T}$ lies in~$\mathfrak{P}$.
\end{remark}

\begin{proof} 
The proof is similar to that of Theorem \ref{teo:fullresult}, but now the main new ingredient is the explicit Corollary \ref{cor:idealeintero}. 

Let $\mathfrak{P}$ be a prime ideal  of $\calo_K$ of norm
$\mathcal{N}(\mathfrak{P})=q>c(M,K)$. We now construct a suitable $\mathfrak{P}$-adic $\calT$-floor function $s \colon K_{\mathfrak{P}} \rightarrow K$  such that the type $(K,\mathfrak{P},\calT, s)$  satisfies the CFF property. 

 Let $\calC$ be the  class of $\mathfrak{P}$ in $\mathrm{Cl}(K)$  and let $\mathfrak{A}=\mathfrak{A}(\calC)$ be the integral ideal in $\calC$ from the statement of Corollary \ref{cor:idealeintero}. 

Let $\eta\in K_{\mathfrak{P}}$. If $\eta=0$, we set $s(\eta)=0$.

 Otherwise, by strong approximation, the coset of $\eta$ modulo $\mathfrak{P}\calo_{\mathfrak{P}}$ contains an element $\alpha'\in K$ such that $|\alpha'|_v\leq 1$ for every non-Archimedean $v\in \calM_K\setminus\{w_0\}$, where $w_0$ is the place of $K$ corresponding to $\mathfrak{P}$.   

Also, by Lemma \ref{lem:effective-5.5}, there exists $\gamma_{\mathfrak{P}}\in K$ such that $|\sigma(\gamma_{\mathfrak{P}})| \leq T_{0} \sqrt[d]{|N_{K/\QQ}(\gamma_\mathfrak{P})|}$ for every $\sigma\in \Sigma$ (where $\Sigma$ is, as usual, the set of embeddings of $K$ in $\CC$) and $\gamma_{\mathfrak{P}}\mathfrak{A}=\mathfrak{P}$. In particular, 
\begin{equation}
    \label{eqn:NgammaNP}
    |N_{K/\QQ}(\gamma_\mathfrak{P})|\leq q.
\end{equation}

By Corollary \ref{cor:idealeintero}, setting $\xi=\frac{\alpha'}{\gamma_{\mathfrak{P}}}$,  there exist $j\in\calT$ and $\tau\in \mathfrak{A}$ such $|\sigma(j\xi - \tau)|<\epsilon$ for every $\sigma\in\Sigma$. 
We finally set \[s(\eta)=\gamma_{\mathfrak{P}} \left (\xi- \frac\tau j\right).\] 
One can easily verify, in the same way as in the proof of Theorem \ref{teo:fullresult}, that the function $s:K_\mathfrak{P}\rightarrow K$ so constructed satisfies all conditions in Definition \ref{def-btff}. Notice that condition~\ref{item-i} follows from Remark~\ref{rem:monsterVSM^d}.

The proof now goes on by showing that we can apply Theorem \ref{teo:FCcondition} to deduce that the type $(K,\mathfrak{P},\calT, s)$ satisfies the $\CFF$ property, i.e.\ that for every $a\in s(K_{\mathfrak{P}})$ such that $|a|_{w_0}>1$ we have 
\begin{equation}\label{nu-tau-a-2} {|a|_{w_0}^{-d_{w_0}}}{\prod_{\sigma\in\Sigma}\theta(\sigma(a))\prod_{w\in \mathcal{M}_K^0\setminus\{w_0\}} \max(|a|_w,1)^{d_w}} <\epsilon ', \end{equation}
for some   $\epsilon'\in (0,1)$, where $\theta$ was defined in \eqref{def-theta}.

Let us set $a=s(\eta)$ where $\eta\in K_{\mathfrak{P}}$ is as above. 
As already noticed in the proof of Theorem~\ref{teo:fullresult} one has:
\begin{equation}\label{bb-sigma-prod}
\prod_{\sigma\in\Sigma}\theta(\sigma(a))<|N_{K/\QQ}(a)|+ \sum_{S\subsetneq \Sigma} \prod_{\sigma\in S} |\sigma(a)|.
\end{equation}
Now notice that for every $\sigma\in \Sigma$ one has
    \begin{equation*}|\sigma(a)|< \epsilon \frac {|\sigma(\gamma_{\mathfrak{P}})|}{j}\leq  \epsilon \frac{T_0 \sqrt[d]{q}}{j}.
    \end{equation*}
Hence, for every subset $S\subsetneq \Sigma$,

\begin{equation}
    \label{b-a-sigma2}
\prod_{\sigma\in S} |\sigma(a)|\leq \left(\frac{\epsilon T_0}{j}\right)^{|S|} \sqrt[d]{q^{|S|}}
\end{equation}
    
  and
    \begin{equation}\label{b-N(a)-sigma2}|N_{K/\QQ}(a)|<\frac {\epsilon^d|N_{K/\QQ}(\gamma_{\mathfrak{P}})|} {|N_{K/\QQ}(j)|}\leq \frac {\epsilon^d q} {j^d},\end{equation}
    where the latter inequality follows from~\eqref{eqn:NgammaNP}.

Plugging \eqref{b-a-sigma2} and \eqref{b-N(a)-sigma2} into \eqref{bb-sigma-prod} we obtain

\begin{align} \prod_{\sigma\in\Sigma}\theta(\sigma(a))&\leq \frac{\epsilon^d q}{j^d} + \sum_{i=0}^{d-1} {d\choose i} {\left (\frac{\epsilon T_0} j\right )^{i} }\sqrt[d]{q^{i}} \nonumber \\
&= 
\frac{\epsilon^d q}{j^d} +\left[\left (\frac{\epsilon T_0\sqrt[d]{q}} j+1\right)^d-q\left(\frac{\epsilon T_0} j\right)^d\right ] \nonumber \\
& =  
\frac{\epsilon^d q}{j^d} +\frac{  q}{j^d}\left[\epsilon^dT_0^d\left(\left (1+\frac {j}{{\epsilon T_0\sqrt[d]{q}}} \right)^d-1\right )\right ] \nonumber\\
&=\frac{q}{j^d}\cdot \epsilon^d\left(1+T_0^d\left(\left (1+\frac {j}{{\epsilon T_0\sqrt[d]{q}}} \right)^d-1\right )\right)\nonumber\\
&\leq \frac{q}{j^d}\cdot \underbrace{\epsilon^d\left(1+T_0^d\left(\left (1+\frac {M}{{\epsilon T_0\sqrt[d]{q}}} \right)^d-1\right )\right)}_{\epsilon'}. \label{eqn:thetaepsilon'}   
\end{align}
 
The quantity $\epsilon'$ tends to $0$ as $q$ tends to infinity. Furthermore, the choice of $q$ in \eqref{eq:mostro} ensures that $\epsilon'<1$.

Let $w\not= w_0$ be a non-Archimedean place of $K$. By the choice of $\alpha'$ one has $|a|_w\leq \frac 1 {|j|_w}$ and $j\in\calo_K\setminus \mathfrak{P}$ by Remark~\ref{rem:monsterVSM^d}. So, using also the product formula, we get 
\begin{equation}\label{bound-factor2-2}
 \prod_{\mathclap{w\in\calM_K^0\setminus\{w_0\}}}\max(|a|_w,1)^{d_w} \leq \prod_{\mathclap{w\in\calM_K^0\setminus\{w_0\} }}\max\left(\frac {1}{|j|_w},1\right)^{d_w} = \prod_{w\in \calM_K^0} \frac {1} {|j|_w^{d_w}}=|N_{K/\QQ}(j)|=j^d.
\end{equation}

Since by hypothesis $|a|_{w_0}>1$, then $|a|_{w_0}^{d_{w_0}}\geq |a|_{w_0}\geq q$. 
This, together with \eqref{eqn:thetaepsilon'} and \eqref{bound-factor2-2}, implies that the left-hand side   in \eqref{nu-tau-a-2} is always at most $\epsilon'<1$, hence Theorem \ref{teo:FCcondition} applies to conclude.
\end{proof}

\section{Explicit examples and possible refinements}
\label{sec:examples}
In this section, we apply Theorem~\ref{teo:fullresult-explicit} and compute the involved constants in some explicit examples. We also consider a general strategy that can be used to (possibly) reduce the values of $M$ and $c(M,K)$.

Let us start with the field $K=\mathbb{Q}(\sqrt{14})$. This field has received attention by many authors since it is the real quadratic field of smallest discriminant which is  Euclidean ~\cite{harper_2004}, but not norm Euclidean~\cite{MurtyClark+1995+151+162}. Performing the computations with
$\Delta = 56$ and $r_2=0$, we find that  $M = \Delta/2 = 28$,
 and, from \eqref{choice-M}, \[\epsilon = \sqrt{\sqrt{\Delta}/M} = 1/\sqrt[4]{14} \approx 0.516973.\] Since the unit rank of $K$ is $1$, the naive upper bound for the covering radius given in Remark~\ref{rem:naiveCoveringRadius} is, in fact, exactly equal to the covering radius, so $\rho \approx 1.70$ and $T_0 \approx 5.48$. Finally, the constant $c(M,K)$ from Theorem~\ref{teo:fullresult-explicit} equals $48896$.

We will now see how the value of $M$ can be decreased using a result by Bedocchi~\cite[page 202]{Bedocchi1985}: he proved that the elements
$a,b\in \mathbb{Z}[\sqrt{14}]$ such that $b\neq 0$ and $|N_{K/\mathbb{Q}}(a-bq)|\geq |N_{K/\mathbb{Q}}(b)|$ for all $\tau\in \mathbb{Z}[\sqrt{14}]$ are exactly those of the form $a=(1+\sqrt{14}+2s)t$ and $b=2t$ for some  $s,t\in \mathbb{Z}[\sqrt{14}]$. Equivalently, there exists $\epsilon<1$ with the following property: for any element $\xi=a/b \in \mathbb{Q}(\sqrt{14})$ there exists $j \in \{1,2\}$ and $\tau \in \mathcal{O}_K$ {such that} $N_{K/\mathbb{Q}}(j\xi -\tau)\leq \epsilon$. This is in fact the property needed to construct the floor function in the proof of Theorem~\ref{teo:fullresult-explicit}, since having class number $1$ allows us to choose $\mathfrak{A} = \mathcal{O}_K$ as a representative of the ideal class of $\mathfrak{P}$. Therefore, we can actually take $M=2$.

In order to conclude the computation of $c(M,K)$ for this new value of $M$, we still need to compute the exact value of $\epsilon$ -- which in this case is the \emph{second Euclidean minimum} of $K$. This has been done by Bedocchi in a later work~\cite[Proposition 3.8]{Bedocchi1989b}, where he found $\epsilon=31/32$. Plugging this value in~\eqref{eq:mostro} yields $c(M,K) =119008$. So, for this example, we have obtained a smaller set of denominators at the price of a way larger constant.

In general, however, it is reasonable to expect -- in the light of Remark~\ref{rem:monsterVSM^d} -- that smaller values of $M$ might lead to smaller values of $c(K,M)$, and the strategy used to refine $M$ in $\mathbb{Q}(\sqrt{14})$ can be indeed generalized to a wider family of number fields. To this end, consider
Table~\ref{tab:constants}, which displays various values of 
$M$ for different number fields together with the corresponding constants $c(M,K)$ obtained applying Theorem~\ref{teo:fullresult-explicit}. 
\begin{table}[ht]    \centering
   \begin{tabular}{r|c|c|c|c|c}
     Minimal polynomial of $\alpha$ & \#  $\mathrm{Cl}(K)$& Signature &$|\Delta|$ &$M$ & $c(M,K)$  \\
     \hline
    $ x^{3} - x^{2} - 2 x + 1 $ & $1$ &$(3,0)$& $7^2$ & $11$ & $40926432$ \\
    $ x^{3} - 3 x - 1 $ & $1$ &$(3,0)$& $3^4$ & $18$ & $187030596$ \\
    $ x^{3} - x^{2} - 3 x + 1 $ & $1$ &$(3,0)$& $2^{2}\cdot 37$ & $33$ & $2446455004$ \\
    $x^{3} - x^{2} - 6 x + 1
$ & $1$ & $(3,0)$ & $5 \cdot 197$ & $219$ & $2626003081902$\\
    $ x^{4} - x^{3} + 2 x - 1 $ & $1$ &$(2,1)$& $5^{2}\cdot 11$ & $21$ & $208540588019$ \\
    $ x^{4} - x - 1$  & $1$ &$(2,1)$& $283$ & $22$ & $187169288265$ \\
    $  x^{4} - x^{3} + x^{2} + x - 1  $ & $1$ &$(2,1)$& $331$ & $26$ & $165348251296$ \\
\end{tabular}
\label{tab:constants}
\caption{Some values of the constants from Theorem~\ref{teo:fullresult-explicit} for $K=\mathbb{Q}(\alpha)$.}
\end{table}

The choice of these number fields is not random. Specifically, all the selected fields satisfy the following three conditions:
\begin{itemize}
    \item they are non-CM fields,
    \item they have degree at least $3$,
    \item the rank of their unit group is strictly larger than 1 (more precisely, it is $2$).
\end{itemize}

It is proven by Cerri \cite[Theorem 5]{Cerri2006} (see also subsequent work of McGown \cite{McGown2012}) that, if a field $K$ satisfies the three conditions above, then the number of cosets in $K/\calo_K$ having Euclidean minimum greater or equal than $\epsilon$ is finite for every $\epsilon>0$. Suppose that we define $\mathcal{T}_\epsilon$ as the set of the norms of the denominators appearing in these cosets: in this way, for every $\xi \in K$, there exist $j \in \mathcal{T}_\epsilon$ and $\tau \in \mathcal{O}_K$ satisfying $N_{K/\mathbb{Q}}(j \xi - \tau)<\epsilon$. 
Thus, as long as the class number of $K$ is $1$ (that is, as long as we further require $\mathfrak{P}$ to be principal), the proof of Theorem~\ref{teo:fullresult-explicit} can be adapted to this choice of $\mathcal{T}_\epsilon$ by simply setting $M=\max(\mathcal{T}_\epsilon)$.

\par The restriction on $\mathfrak{P}$ being principal, highlighted above, can be in fact dropped thanks to a suitable generalization of~\cite[Theorem 5]{Cerri2006}. To this end, for each $x \in K$ we consider the \emph{Euclidean minimum of $x$ with respect to  (the class of) $\mathfrak{A}$}
\[m_{\mathfrak{A}}(x)=\frac{1}{\mathcal{N}(\mathfrak{A})}\inf_{\tau \in \mathfrak{A}}N_{K/\mathbb{Q}}(x-\tau),\]
as defined in~\cite{McGown2012}. The map  $m_{\mathfrak{A}}$ extends to an upper semicontinuous function on $\mathbb{R}^d$, whose properties in Cerri's setting (i.e.\ when $\mathfrak{A}=\calo_K$) can be generalized as follows:
\begin{proposition}
    \label{prop:CerriGen}
    Let $K$ be a non-CM number field of degree at least $3$ and rank of unit group strictly larger than $1$, and let $\mathfrak{A}$ be an integral ideal of $K$. Let $\lambda$ be the embedding $K\rightarrow \mathbb{R}^d$ defined in Section~\ref{sec-embeddings}. For each $\epsilon>0$, the set
    \[\mathcal{S}_\epsilon=\{x \in \mathbb{R}^d/\lambda(\mathfrak{A}) \mid m_{\mathfrak{A}}(x)\geq \epsilon\}\]
    is finite.
\end{proposition}
\begin{proof}
The proof is based on  \cite[Theorem\,2]{Cerri2006}, where it is proven that if $\mathbb{T}$ is 
an $n$-dimensional torus  and $\Sigma$ is a commutative semigroup  of endomorphisms of $\mathbb{T}$, then every proper closed $\Sigma$-invariant subset of $\mathbb{T}$ is finite if and only if the the following three conditions hold:
\begin{enumerate}
\item there exists  $\sigma\in \Sigma$ such that, for every integer $n$, the characteristic polynomial of $\sigma^n$ is irreducible;
\item every eigenvector of $\Sigma$ possesses an eigenvalue of modulus strictly bigger than~1;
\item $\Sigma$ has at least two elements which are rationally independent.
\end{enumerate}
    Consider now the torus $\mathbb{T}=\mathbb{R}^d/\lambda(\mathfrak{A})$ and $\Sigma=\lambda(\calo^\times)$, where $\Sigma$ is seen as a commutative semigroup of endomorphisms on $\mathbb{T}$ via the action given by the multiplication by units. 
    
    In this setting one can check that the three conditions above are satisfied following \emph{verbatim} the proof of the first part ~\cite[Theorem\,5]{Cerri2006}, replacing the ring of integers of $K$ with $\mathfrak{A}$. This concludes the proof as the set $\mathcal{S}_\epsilon$ is clearly $\Sigma$-invariant.
\end{proof}

To sum up, one can proceed as follow:
\begin{itemize}
    \item List all the \emph{successive minima} $\epsilon_1, \dots, \epsilon_{n+1}$ until $\epsilon_{n+1}<1$.
    \item Compute $\mathcal{T}_{\epsilon_n}$.
    \item Set $\epsilon=\epsilon_n$ and $\mathcal{T}=\mathcal{T}_{\epsilon_n}$.
\end{itemize}
When $\mathfrak{P}$ is principal, these steps can be performed using an algorithm by Lezowski~\cite{Lezowski2014}.\footnote{Code available at \url{https://www.math.u-bordeaux.fr/~plezowsk/recherche.php}.} It would be interesting to generalize Lezowski's algorithm for computing the successive minima of a non-trivial ideal class, but we are not aware of such a result.

\par Finally, we remark that there are of course cases in which the three conditions required by \cite[Theorem 5]{Cerri2006} are not met. For instance, in degree $3$, while the non-CM assumption is always met, the condition that the rank of the unit group be strictly greater than $1$ (which, in this case, is equivalent to the field having signature  $(3,0)$) is presumably rare for the splitting field of polynomials of degree 3 produced via certain random processes (see for instance \cite{EdelmanKostlan1995}).

\section{Related work: continued fractions without floor functions}
\label{sec:related}

Throughout the previous sections, we managed to construct types $\tau=(K,\mathfrak{P},\calT, s)$ that satisfy the $\CFF$ property. In other words, we managed to encode any element $\alpha \in K$ as a finite sequence $[a_0, a_1, \dots,a_k]$ of elements belonging to some smaller set, namely
\[s(K)\subseteq\left\{\frac{a}{j}\;\middle |\; a \in \mathcal{O}_{\{\mathfrak{P}\}},\, j \in \mathcal{T}\right\},\]
so that
\begin{equation}
\label{eqn:CFexp}
   \alpha= {{a_0 + \cfrac{1}{{a_1 +  \cfrac{1}{{\ddots + \cfrac{1}{{a_{k}}}}}}}}}.
\end{equation}
The only part of our construction which is not fully explicit -- and therefore not easy to convert into an actual CF-algorithm -- is the $\mathcal{T}$-floor function $s$. It is then natural to ask whether an expression of the form~\eqref{eqn:CFexp} can be computed without resorting to any floor function.  The answer is positive: there are some strong results on the existence of short expansions of the form~\eqref{eqn:CFexp}. This, of course, comes with a price, which essentially amounts to losing the $\mathfrak{P}$-adic convergence of CF-expansions of the elements in $K_{\mathfrak{P}}\setminus K$. In this section we will briefly survey these short expansions and compare them with our floor-function-based CF-expansions.

\subsection{Continued fractions and division chains}
Most of the results we will consider are in fact referred to division chains rather than continued fractions. It is well-known that these notions are closely related, as we will review in this section.
\par Let $K$ be a number field and let $R$ be a subring of $K$. Given $a,b \in R$, a \emph{division chain of length $k$} for the pair $a,b$ is a sequence of quotients $q_1, \dots, q_k \in R$ and residues $r_1, \dots, r_{k}\in R$ such that
\begin{align}
\label{eqn:divChain}
   \begin{split}
        a&= q_1 b + r_1\\
    b&= q_2r_1 + r_2\\
    &\vdots\\
    r_{k-2} &=q_kr_{k-1} + r_k.
   \end{split}
\end{align}
A division chain is \emph{terminating} if $r_k=0$.
\par Suppose that the pair $a,b$ has a terminating division chain of length $k$. Then, using~\eqref{eqn:divChain}, we can write
\begin{equation*}
    \frac{a}{b}=q_1+ \cfrac{1}{{b}/{r_1}}= q_1+ \cfrac{1}{q_2+\cfrac{1}{{r_1}/{r_2}}}= \dots  = q_1 + \cfrac{1}{q_2 + \cfrac{1}{\ddots + \cfrac{1}{q_{k}}}}.
\end{equation*}
In other words, $[q_1, \dots, q_n]$ is a continued fraction expansion of $a/b$. We stress that, in general, it is not obtained via algorithm~\eqref{eqn:CFalg} (in particular, the elements $q_i$ may not lie in the image of any floor function on $K$).
\begin{remark}
    The $\CFF$ property for $(K,R)$ is also related to the fact that the group $\SL_2(R)$ (or $\SL_n(R)$) can be generated by elementary matrices (see \cite{OMeara1965, Cohn1966, CossuZanardoZannieri2018}).
\end{remark}

We can now rephrase the CFF property in the language of division chains.
\begin{definition}
    Let $K$ be a number field and $R$ be any subring of $K$.
    \begin{itemize}
        \item $R$ satisfies the \emph{Euclidean chain condition}~\cite{OMeara1965} if every pair of elements $a,b \in R$ admits a terminating division chain.
        \item $R$ is \emph{$k$-stage Euclidean} if it satisfies the Euclidean chain condition and each division chain has length at most $k$~\cite{Cooke1976a}.
    \end{itemize}
\end{definition}
It is well-known that the Euclidean chain condition, which is exactly the CFF property seen in the context of division chains, holds for $(\QQ,\ZZ)$ and, more generally, for $(K,R)$ when $R$ is an Euclidean ring and $K=\mathrm{Frac}(R)$.
\par The property of being $k$-stage Euclidean, on the other hand, seems much harder to satisfy, since it dictates the maximum length allowed for division chains. Nevertheless, in many relevant cases it does hold for remarkably small values of $k$, as we will see in Section~\ref{subsec:kstageResults}.

\subsection{Euclidean chain condition and ideal class group}

Let us focus again on the case $R=\mathcal{O}_S$, where $S$ is a finite set of places of $K$. An essential prerequisite for the Euclidean chain condition (and therefore, in particular, the $\CFF$ property) arises from the very structure of the ideals of $\mathcal{O}_S$. It is proven in \cite[Prop.\,2]{Cooke1976a} that, if $(K,\mathcal{O}_S)$ satisfies the Euclidean chain condition, then $\calo_S$ is a unique factorization domain. In particular, if the Euclidean chain condition holds for $(K,\calo_K)$, then the class number of $K$ is $1$. Furthermore, when the group of units $\calo_K^\times$ is infinite, the converse also holds~\cite[Theorem 1]{Cooke1976a}. \par

A more general necessary condition is motivated by the following result, which is based on~\cite[Cor.\,7.2]{CapuanoMurruTerracini2022}.
\begin{proposition}\label{prop:classgroup}
Let $K$ be a number field, let $S$ be a finite set of places of $K$, and assume that 
the pair $a,b \in \calo_S$ has a terminating division chain with quotients in $\calo_S$. Then, the class in ${\mathrm{Cl}}(K)$ of the fractional ideal $(a,b)$ belongs to the subgroup generated by the classes of the ideals in $S$.
\end{proposition}
\begin{proof} 
Let $a_0,\ldots,a_n\in\calo_{S}$ be the quotients in the division chain of the pair $a,b$.
One can prove, by induction on $n$, that the continued fraction $[a_0, \dots, a_n]=a/b$ is equal to the ratio of two elements $A_n, B_n \in \calo_S$, which can be defined recursively as
\begin{align*} 
A_{-1}=1,\ &A_0=a_0,\ A_{n}= a_nA_{n-1}+A_{n-2},\\
B_{-1}=0,\ &B_0=1,\ \ B_{n}= a_nB_{n-1}+B_{n-2},
\end{align*}
for $n\geq 1$. Moreover, $A_n$ and $B_n$ are coprime since they satisfy $A_n B_{n-1} + A_{n-1}B_n=1$. Therefore, the ideal $(A_n, B_n)=(a,b)$ belongs to the kernel of the natural projection of $\mathrm{Cl}(K)$ into the class group of $\calo_S$. The kernel of this projection is generated by the classes of the ideals in $S$~\cite[452]{neukirchEtAl:cohomology}.

\end{proof}

\begin{corollary}\label{cor: condizionesuffperCFF}
 Suppose that $R$ satisfies the Euclidean chain condition and let $S$ be the set of the prime ideals $\mathfrak{P}$ of $\calo_K$ such that $v_\mathfrak{P}(a)<0$ for some $a\in R$. Then  ${\mathrm{Cl}}(K)$ is generated by the classes of  the ideals in $S$.
\end{corollary}
Corollary \ref{cor: condizionesuffperCFF} is one of our motivations for introducing extraneous denominators: classic $\mathfrak{P}$-adic continued fractions may have the $\CFF$ only when the class group is generated by the class of $\mathfrak{P}$, while $\mathfrak{P}$-adic continued fractions of type $\tau=(K,\mathfrak{P},\calT, s)$ may have the $\CFF$ when the class group is generated by the classes of $\mathfrak{P}$ \emph{and} the ideals generated by $\calT$.

\subsection{Some results on \texorpdfstring{$k$}{k}-stage Euclideanity}
\label{subsec:kstageResults}
In terms of continued fractions, saying that $\mathcal{O}_S$ is $k$-stage Euclidean means that every element in $K$ can be expressed as a continued fraction with length $k$ and partial quotients in $\mathcal{O}_S$. It is surprising that, if $\mathcal{O}_S$ contains enough units and $S$ meets the necessary condition of Corollary~\ref{cor: condizionesuffperCFF}, one can set $k$ to be as low as $5$~\cite{morganEtAl:boundedGenSL2}.
\begin{theorem}
    Assume that there are infinite units in $\mathcal{O}_S$. Let $a,b$ be coprime elements in $\mathcal{O}_S$. Then the pair $a,b$ has a division chain of length $5$. 
\end{theorem}
The same result was proven -- under GRH -- in~\cite[Thms.\,2.2 and 2.9]{CookeWeinberger1975}, with the further observation that $a/b$ has a CF-expansion of length $4$ if $S$ contains at least one finite place. The following algorithm, introduced by Cooke, Lenstra and Weinberger, can be used to compute these CF-expansions whenever a finite principal place $\mathfrak{P}=(p)$ belongs to $S$:
\begin{itemize}
    \item Choose $q_1$ and $r_1$ such that
    \begin{align*}
        a = q_1b + r_1 &&\text{and} &&|r_1|_p=1.
    \end{align*}
    This ensures that the next step can be done within a finite amount of time.
    \item Check each principal prime ideal $(p')$ not in $S$ until the  following conditions are met:
    \begin{itemize}
        \item $b \equiv p' \mod r_1$,
        \item $r_1\equiv u \mod p'$, where $u$ is a unit.
    \end{itemize}
    We remark that it is sufficient to check the latter condition only on a finite number of units, since $\mathcal{O}_S/(p')$ is finite. 
    \item For the $p'$ found at the previous step, complete the chain
    \begin{align*}
        b&=\underbrace{\left(\frac{b-p'}{r_1}\right)}_{q_2}r_1 + \underbrace{p'}_{r_2},\\
        r_1 &= \underbrace{\left(\frac{r_1-u}{p'}\right)}_{q_3}p'+\underbrace{u}_{r_3}\\
        r_2 &= \underbrace{{u^{-1}r_2}}_{q_4}u
    \end{align*}
\end{itemize}
The above strategy can be implemented and allowed us to successfully find, for example, a $5$-adic CF-expansion of $7/3$ in the field $\mathbb{Q}(z)$ where $z$ is a root of $x^3+x+1$:
\begin{equation*}
    \frac{7}{3}=-1+\cfrac{1}{\frac{1-z}{5}+\cfrac{1}{-12z^2+ 6z - 15+\cfrac{1}{2z + 1}}}.
\end{equation*}
In general, though, the algorithm runs indefinitely due to the hardness of finding $p'$. However, as far as we know, there is still no efficient algorithm for calculating the corresponding division chains. This is a further reason for considering the floor-function-based approach, which seems easier to implement and, as already remarked, has the further property of providing converging approximations of the elements in $K_\mathfrak{P}\setminus K$.

\section*{Acknowledgements}
The authors thank Denis Simon for suggesting a more compact proof of Proposition~\ref{prop:classgroup}. They also thank the anonymous reviewers for their feedback and comments.

They all acknowledge the IEA (International Emerging Action) project PAriAlPP (Probl\`emes sur l'Arithm\'etique et l'Alg\`ebre des Petits Points) of the CNRS for the financial support. 

Laura Capuano, Marzio Mula and Lea Terracini  are members of the INdAM group GNSAGA. 

Sara Checcoli’s work has been also supported by the French National Research Agency in the framework of the Investissements d’avenir program (ANR-15-IDEX-02).

\printbibliography

\end{document}